\documentclass[12pt, leqno]{article}

\usepackage{amsmath,amsthm}
\usepackage{amssymb}

\usepackage{enumerate}

\usepackage{graphicx}

\newtheorem{thm}{Theorem}[section]
\newtheorem{theorem}{Theorem}[section]
\newtheorem{cor}[thm]{Corollary}
\newtheorem{lem}[thm]{Lemma}
\newtheorem{prop}[thm]{Proposition}

\theoremstyle{definition}

\numberwithin{equation}{section}

\usepackage{amsfonts}
\usepackage{mathrsfs}
\usepackage{enumerate}
\usepackage{longtable}
\usepackage[citecolor=red]{hyperref}

\usepackage{url}
\usepackage{multirow}

\begin{document}

\title{Non-vanishing Fourier coefficients of modular forms}

\author{Peng Tian\\
Department of Mathematics\\
East China University of Science and Technology\\
200237, Shanghai, P. R. China\\
and\\
Hourong Qin\\
Department of Mathematics\\
Nanjing University\\
210093, Nanjing, P. R. China}


\maketitle





\renewcommand{\thefootnote}{\arabic{footnote}}
\setcounter{footnote}{0}

\begin{abstract}
In this paper, we generalize D. H. Lehmer's result to give a sufficient condition for level one cusp forms $f$ with integral Fourier coefficients such that the smallest $n$ for which the  coefficients $a_n(f)=0$ must be a prime. Then we describe a method to compute a bound $B$ of $n$ such that $a_n(f)\ne0$ for all $n<B$. As examples, we achieve the explicit bounds $B_k$ for the unique cusp form $\Delta_{k}$ of level one and weight k with $k=16, 18, 20, 22, 26$ such that $a_n(\Delta_k)\ne0$ for all $n<B_k$.

\end{abstract}

\section{Introduction}

In 1947, D. H. Lehmer conjectured that Ramanujan's tau function $\tau(n)$ is non-vanishing for all $n$. In \cite[Theorem 2]{lehmer} he proved that the smallest $n$ for which $\tau(n)=0$ must be a prime and showed that $\tau(n) \neq 0$ for all $n<3316799$ by observing the congruences of $\tau(p)$ for certain primes $p$. 

J-P. Serre \cite{serrelehmer} summed up the congruences of $\tau(p)$ for exceptional primes of $\tau(p)$(for the definition see Section \ref{sec:representations}) and showed that if $\tau(p)=0$ for a prime $p$, then 

\begin{center}
	$\begin{array}{lll}
	p\equiv -1 \mod 2^{11}3^75^3691, \\
	p\equiv  -1,19,31 \mod 7^2 \ \ \ \ \ \ \ \mathrm{and} \\
	p\equiv \ a \ non$-$square \mod 23. \\
	\end{array}$
\end{center}
He then obtained a bound of 15 digits for Lehmer's conjecture with respect to  $\tau(n)$.

In the book \cite{book}, S. J. Edixhoven, J.-M. Couveignes, et al. proposed a polynomial time algorithm to compute the mod $\ell$ Galois representation $\rho_{f,\ell}$ associated to modular form $f(z)= \sum_{n>0} a_{n} (f) q^{n} \in S_{k}(SL_2(\mathbb{Z}))$, where $q=e^{2\pi iz}$. In practice, this algorithm can be used to approximately evaluate the polynomial $\tilde P_{f,\ell}\in \mathbb{Q}[x]$ of degree $\ell+1$ whose splitting field can be described as the fixed field of the kernel of the associated projective representation $\tilde{\rho}_{f,\ell}$. For a prime number $p\ne \ell$, it can be shown that $a_{p}(f)\equiv 0$ mod $\ell$ if and only if there exists a prime $\mathfrak{p}|p$ of degree $2$ in the number field $K=\mathbb{Q}[x]/(\tilde P_{f,\ell})$. Hence, for $p \nmid Disc(\tilde P_{f,\ell})$, we can verify $a_{p}(f)\equiv 0$ mod $\ell$ by checking whether the projective modular polynomial $\tilde P_{f,\ell}$ has an irreducible factor of degree $2$ over $\mathbb{F}_p$. 

This method has been first applied by J. Bosman \cite[Corollary 7.14]{book} to search a bound for Lehmer's conjecture with respect to  $\tau(n)$. More precisely, he proved that $\tau(n)\ne 0$ for all  $n$ with $n<22798241520242687999.$ So far the record \cite{martin} of the bound is $816212624008487344127999$.

Let $\Delta_k$ denote the unique cusp form of level $1$ and weight $k$ with $k=12,16,18,20,22$ and $26$. In this paper, we will discuss non-vanishing Fourier coefficients of level one cusp forms $f$ with integral Fourier coefficients.  

We first generalize D. H. Lehmer's result to show that, for a newform $f$ of weight $k$ and level one, if it has integral Fourier coefficients and $|a_2(f)|\neq 2^{\frac{k}{2}}$, $|a_3(f)|\neq 3^{\frac{k}{2}}$, then the smallest $n$ for which the Fourier coefficients $a_n(f)=0$ must be a prime. In particular, the cusp form $\Delta_k$ satisfies the conditions and hence the smallest $n$ for which $a_n(\Delta_k)=0$ must be a prime.

In \cite{sd1} and \cite{sd2}, H. P. F. Swinnerton-Dyer shows a method to determine the exceptional primes $\ell$ for $f$ and the congruences of $a_{p}(f)$ modulo the powers of $\ell$. In particular, he explicitly determined almost all the exceptional primes $\ell$ for $\Delta_k$ and for primes $p\ne\ell$ achieved the congruences of $a_p(\Delta_k)$ modulo the powers of $\ell$. We will summarize his congruences and then, for each prime $p$ provided with $a_p(\Delta_k)=0$, obtain the formulations that such $p$ must satisfy. We also give a complete proof for the congruences of $a_p(\Delta_{16})$ modulo $\ell=31$ following an idea proposed in \cite{sd1}. 

Then we do a computer search on primes $p$ satisfying these formulations, and then verify whether $a_{p}(\Delta_k)\equiv 0$ mod $\ell$ by checking whether the projective modular polynomial $\tilde P_{\Delta_{k},\ell}$ has an irreducible factor of degree $2$ over $\mathbb{F}_p$. Consequently, we achieve a large bound $B_k$ of $n$ such that $a_n(\Delta_k)\ne0$ for all $n<B_k$.

\section{The first non-vanishing Fourier coefficient}

In \cite[Theorem 2]{lehmer} D. H. Lehmer proved that the smallest $n$ for which $\tau(n)=0$ must be a prime. Using his method, we can show the Fourier coefficients of level one modular forms may also possess this property. First we give the following formulas for Fourier coefficients.

\begin{lem} 
	Let $f = \sum_{n>0} a_{n} (f) q^{n} \in S_{k}(SL_2(\mathbb{Z}))$ be a newform of level $1$ and weight $k$. Suppose that $a_n(f)\in \mathbb{Z}$. For prime $p$ and integer $n\ge0$, we have 
	\begin{equation} \label{formulaoftri}
	a_{p^n}(f)=p^{\frac{k-1}{2}\cdot n} \cdot \csc \theta_{p} \cdot  sin[(n+1)\theta_{p}], 
	\end{equation}
	where $\theta_{p}\in\mathbb{R}$ and $2\cos\theta_{p} = a_{p}(f)p^{-\frac{k-1}{2}}$.
\end{lem}
\begin{proof}
	For each prime $p$, we have $|a_{p}(f)| \le 2p^{\frac{k-1}{2}}$ and then we can take $\theta_{p}\in\mathbb{R}$ such that 
	$$ a_{p}(f)= 2p^{\frac{k-1}{2}}\cos\theta_{p}. $$
	We know the Fourier coefficients of $f$ satisfy the formulas
	$$ a_{p^{n+2}}(f) = a_p(f)a_{p^{n+1}}-p^{k-1} a_{p^{n}}(f),  \  \  \ for \ p \ and \ n\ge 0. $$
	Now we prove this lemma by induction. 
	The formula (\ref{formulaoftri}) obviously holds for the cases $n=0, 1, 2$. For any integer $n>2$, we suppose this formula also holds for all $m$ with $0\le m \le n-1$. By straightforward calculations  we have 
	$$\cos\theta_p \cdot \sin(n \theta_{p}) - \sin[(n-1)\theta_{p}]= \sin\theta_p \cdot \cos(n\theta_{p}) .$$
	Thus we have 
	\begin{center}
		$\begin{array}{lll}
		a_{p^{n}}(f)& = &  a_p(f)a_{p^{n-1}}-p^{k-1}a_{p^{n-2}}(f)  \\
		& = & 2p^{\frac{k-1}{2}} \cdot \cos \theta_{p} \cdot  p^{\frac{k-1}{2}\cdot n} \cdot \csc \theta_{p} \cdot \sin(n\theta_{p})           \\
		&   & \ \ \ \ \ \ \ \ \ \ \ \ \ \ \ \ \ \  \ \ \ \ \ \ \ - p^{\frac{k-1}{2}} \cdot p^{\frac{k-1}{2}\cdot (n-2)} \cdot \csc \theta_{p} \cdot \sin[(n-1)\theta_{p}]  \\
		& = & p^{\frac{k-1}{2}\cdot n} \cdot \csc\theta_p \cdot ( 2\cos\theta_p \cdot \sin(n\theta_{p}) - \sin[(n-1)\theta_{p}]   )     \\
		& = & p^{\frac{k-1}{2}\cdot n} \cdot \csc\theta_p \cdot ( \cos\theta_p \cdot \sin(n\theta_{p}) + \sin\theta_p \cdot \cos(n\theta_{p}) )    \\
		& = & p^{\frac{k-1}{2}\cdot n} \cdot \csc\theta_p \cdot \sin[(n+1)\theta_{p}].    \\
		\end{array}$ 
	\end{center}
	
\end{proof}

Now we can show 
\begin{theorem} \label{primes}
	Let $k>0$ be an even integer. Let $f = \sum_{n>0} a_{n} (f) q^{n} \in S_{k}(SL_2(\mathbb{Z}))$ be a newform of weight $k$ and level $1$, with $a_n(f)\in \mathbb{Z}$.  Suppose $|a_2(f)|\neq 2^{\frac{k}{2}}$ and $|a_3(f)|\neq 3^{\frac{k}{2}}$. Then the smallest $n$ for which $a_{n} (f)=0$ is a prime.
\end{theorem}
\begin{proof}
	Assume $n_0$ is the  smallest integer for which $a_{n_0} (f)=0$. Since we have 
	$$ a_{mn} (f) = a_{m} (f) a_{n} (f), \ \ \  if \ (m,n)=1,  $$
	it follows that $n_0$ must be a power of a prime. We may suppose $n_0 = p^{r}$ for a prime $p$ and an integer $r\ge 1$. Now we will show $r=1$. 
	
	By Lemma \ref{formulaoftri} we have 
	$$ 0 = a_{p^r}(f)=p^{\frac{k-1}{2}\cdot r} \cdot \csc \theta_{p} \cdot  \sin[(r+1)\theta_{p}],   $$
	where $2\cos\theta_{p} =p^{-\frac{k-1}{2}} a_{p}(f)$.
	It follows that  $\sin[(r+1)\theta_{p}]=0$ and hence $\theta_{p}$ is a rational number of the form
	$$  \theta_p = \pi t/(r+1),      $$
	for some integer $t$. We set $\alpha=2\cos\theta_p = p^{-\frac{k-1}{2}} a_{p}(f)$. Then $\alpha$ is a root of the quadratic polynomial with rational coefficients
	$$P(x)= x^2 -  a_{p}^{2}(f)/p^{k-1}.    $$
	Now suppose $r>1$. Then $a_{p}(f)\ne 0$ and $\alpha$ is not rational. It follows that $P(x)$ is the minimal polynomial of $\alpha$ over $\mathbb{Q}$. Moreover, we know $\alpha=2\cos(\pi\cdot\frac{t}{r+1})$ is an algebraic integer and hence $P(x)\in \mathbb{Z}[x]$, namely $\alpha^2= a_{p}^{2}(f)/p^{k-1}$ is a positive non-square integer. Since $\alpha^2= (2\cos\theta_p)^2\le 4$, we have $a_{p}^{2}(f)/p^{k-1}=2$ or $3$. This implies that the integer $a_{2}(f)=\pm 2^{\frac{k}{2}}$ or $a_{3}(f)\pm 3^{\frac{k}{2}}$. This contradicts the supposition that  $|a_2(f)|\neq 2^{\frac{k}{2}}$ and $|a_3(f)|\neq 3^{\frac{k}{2}}$. Hence $r=1$.

\end{proof}

In particular we have

\begin{cor} \label{primesfordelta}
	Let $\Delta_k$ denote the unique cusp form of level $1$ and weight $k$ with $k=12,16,18,20,22$ and $26$. Then the smallest $n$ for which the Fourier coefficients $a_n(\Delta_k)=0$ must be a prime.
\end{cor}

\section{Galois representations associated to modular forms} \label{sec:representations}

Let $f = \sum_{n>0} a_{n} (f) q^{n} \in S_{k}(SL_2(\mathbb{Z}))$ be a newform of level one with the nebentypus character $\varepsilon_{f}$. Let $K_f$ be the number field which is obtained by adjoining all coefficients $ a_{n}(f)$ to $\mathbb Q$. Let $\ell$  be a prime number and $\lambda$ be a prime of $K_{f}$ lying over $\ell$. Denote $\mathbb{F}_{\lambda}$ the residue field of $\lambda$. Then Deligne \cite{Deligne71} shows that there exists a continuous semi-simple representation
\begin{center}
	$\rho_{f,\lambda}: Gal(\overline{\mathbb{Q}}|\mathbb{Q}) \rightarrow GL_{2}(\overline{\mathbb{F}}_{\ell})$,
\end{center}
which is unique up to isomorphism and which has the property that for primes $p\ne\ell$ one has
\[a_p(f)\equiv \mathrm{tr}(\rho_{f,\lambda}(\mathrm{Frob}_{p})) \mod \ell . \] 

A prime $\ell$ is said to be exceptional if the image of $\rho_{f,\ell}$ does not contain $SL_{2}(\mathbb{F}_{\ell})$.

In the book \cite{book}, S. Edixhoven and J.-M. Couveignes propose a polynomial time algorithm to compute $\rho_{f,\lambda}$. In fact it is equivalent to find the polynomial $P_{f,\ell} \in \mathbb{Q}[x]$ of degree $\ell^2-1$, whose splitting field is the fixed field of ${\rm ker}(\rho_{f,\lambda})$. In the same way, the associated projective representation $\tilde{\rho}_{f,\lambda}$ can be described as the splitting field of a certain polynomial $\tilde P_{f,\ell}\in \mathbb{Q}[x]$ of degree $\ell+1$. Moreover we have

\begin{lem} \label{orbit}
	Let $P(x)\in\mathbb{Q}[x]$ be an irreducible polynomial of degree $\ell+1$. Denote $\tilde{\rho}_{f,\ell}$ the projective representation mod $\ell$ associated to a newform $f$ for a prime $\ell$. Suppose the splitting field of $P(x)$ is the fixed field of ${\rm ker}(\tilde{\rho}_{f,\ell})$ and the Galois group Gal$(P(x))$ of $P(x)$ is isomorphic to $PGL_{2}(\mathbb{F}_{\ell})$. Then a subgroup of $Gal(\mathbb{\overline{Q}|Q})$ fixing a root of $P(x)$ corresponds via $\tilde{\rho}_{f,\ell}$ to a subgroup of $PGL_{2}(\mathbb{F}_{\ell})$ fixing a point of $\mathbb{P}^{1}(\mathbb{F}_{\ell})$.
\end{lem}
\begin{proof}
	A subgroup of $Gal(\mathbb{\overline{Q}|Q})$ fixing a root of $P(x)$, by the canonical map $Gal(\overline{\mathbb{Q}}/\mathbb{Q})\twoheadrightarrow Gal(P(x))$,  corresponds to a subgroup of $Gal(P(x))$ of index deg$(P(x))=\ell+1$, and thus its  image via $\tilde{\rho}_{k,\ell}$ is a subgroup of $PGL_{2}(\mathbb{F}_{\ell})$ of index $\ell+1$. Then it follows from \cite[Lemma 7.3.2]{book}.
\end{proof}

In practice, J. Bosman first did explicit computations and obtained $\tilde P_{f,\ell}$ for modular forms $f$ of level 1 and of weight $k\le 22$, with $\ell \leq 23$. Recently, others improved the algorithm and computed the polynomials for more cases. See \cite{1} and \cite{martin} for instance.\\
Since the polynomial $\tilde P_{f,\ell}\in \mathbb{Q}[x]$ is monic and irreducible, one can check if $a_p(f)$ is non-vanishing for prime $p$ as follows. We denote the number field of $\tilde P_{f,\ell}$ by  $K=\mathbb{Q}[x]/(\tilde P_{f,\ell})$. For each $p\ne\ell$ we have
\[a_p(f)\equiv \mathrm{tr}(\rho_{f,\lambda}(\mathrm{Frob}_{p})) \mod \ell . \] 
Then it follows from \cite[Lemma 7.4.1]{book} that $a_{p}(f)\equiv 0$ mod $\ell$ if and only if the action of $\rho_{f,\ell}(Frob_{p})$ on $\mathbb{P}^{1}(\mathbb{F}_{\ell})$ has an orbit of length 2, and by Lemma \ref{orbit}, if and only if $Frob_{p}^{2}$ fixes a root of $\tilde P_{f,\ell}$, but $Frob_{p}$ does not. This means that there exists a prime $\mathfrak{p}|p$ of degree $2$ in $K$, and then for $p \nmid Disc(\tilde P_{f,\ell})$, we can verify the condition by checking whether $\tilde P_{f,\ell}$ has an irreducible factor of degree $2$ over $\mathbb{F}_p$.

\section{Congruences for the coefficients of $\Delta_k$} \label{sec:congruences}

In this scetion we suppose that $f = \sum_{n>0} a_{n} (f) q^{n} \in S_{k}(SL_2(\mathbb{Z}))$ be a newform of weight $k$ and level $1$, with $a_n(f) \in \mathbb{Z}$. 
From the previous section we know, for a prime $\ell$, there exists a Galois representation $\rho_{f,\ell}$ associated to $f$.

In \cite{sd1} and \cite{sd2}, H. P. F. Swinnerton-Dyer observes the image of  $\rho_{f,\ell}$ and shows a method to determine the exceptional primes $\ell$ for $f$ and the congruences of $a_{p}(f)$ modulo the powers of $\ell$. 

He first shows that if and only $\ell$ is an exceptional prime, the Fourier coefficients have congruences modulo $\ell$, which must be one of the three cases as follows
\begin{equation} \label{threetypesofcongruences}
\begin{array}{cll}
(i)    &  a_p(f)\equiv p^{m}+p^{k-1-m} \ mod \ \ell \ \  for \ all \ primes \ p\ne \ell;\\
(ii)   &  a_p(f)\equiv 0 \mod \ell \ if  \ p \  is  \ not \ a \  quadratic \  residue \ modulo \ \ell;\\
(iii)  &  p^{1-k}a_p(f)^2 \equiv 0,1,2 \ or \ 4 \ mod \ \ell \  \ for \ all \ primes \ p\ne \ell.\\
\end{array}
\end{equation}
Then he gives the bounds of all the exceptional primes and connections of $\ell$, $m$ and $k$, which allows us to explicitly determine almost all of the exceptional primes. According to (\ref{threetypesofcongruences}), there are three types of exceptional primes and only the primes of type (iii) cannot be explicitly computed. However, one can find a finite set of primes which contains all primes of this type. For each exceptional prime $\ell$ of type (i), together with the values of  $a_p(f)$ for certain $p$, one can compute $m$ and $N\ge1$ in
\begin{equation} \label{congruenceformula}
a_p(f)\equiv p^{m}+p^{k-1-m} \mod \ell^N.
\end{equation}

Now let  $\Delta_k$ be the unique cusp form of level $1$ and weight $k$ with $k=12,16,18,20,22$ and $26$. We summarize Swinnerton-Dyer's  results and obtain the following tables which show the values of $m$ and $N$ in (\ref{congruenceformula}) for the exceptional primes $\ell$  type (i).

\newpage

\makeatletter\def\@captype{table}\makeatother

\begin{center}
	\begin{tabular}{|c|c|c|c|c|c|c|}
		\hline
		$\ell$ & \ \ 3 \ \ & \ \ 5 \ \  & \ \  7  \ \ & \ \ 11 \ \ & \ \ 3617 \ \  \\  
		
		\hline

		\multirow{2}{*}{$N$}  & 5 if $(\frac{p}{3})=1$  \ \ & \multirow{2}{*}{2} &  \multirow{2}{*}{3} \ \ &  \multirow{2}{*}{1} &  \multirow{2}{*}{1}\\

		&  6 if $(\frac{p}{3})=-1$ &  &   &  & \\
		
		\hline
		
		$m$ & 174 & 17 &  85 & 1 &0 \\
		
		\hline
		
	\end{tabular}
\end{center}
\caption{k=16}

\vspace{1.2cm}

\makeatletter\def\@captype{table}\makeatother

\begin{center}
	\begin{tabular}{|c|c|c|c|c|c|c|c|}
		\hline
		$\ell$   & \ \ 3 \ \ & \ 5  \  & \ \  7  \ \ & \ \ 11 \ \ & \  13  \  & \  43867  \   \\   
		
		\hline
		
		\multirow{2}{*}{$N$}   & 5 if $(\frac{p}{3})=1$ \ \ & \multirow{2}{*}{3} & 1 if $(\frac{p}{7})=1$  & 1 if $(\frac{p}{11})=1$    & \multirow{2}{*}{1} & \multirow{2}{*}{1}\\

		&  6 if $(\frac{p}{3})=-1$  &  &  2 if $(\frac{p}{7})=-1$   &   2 if $(\frac{p}{11})=-1$  &&  \\
		
		\hline
		
		$m$ & 386 & 22 &  1 & 1 &1 &0\\
		
		\hline
		
	\end{tabular}
\end{center}
\caption{k=18}

\vspace{1.2cm}

\makeatletter\def\@captype{table}\makeatother

\begin{center}
	\begin{tabular}{|c|c|c|c|c|c|c|c|c|}
		\hline
		$\ell$  & \  3 \ & \ 5 \  & \   7  \ & \ 11 \ &  \ 13  \  & \ 283  \ &  \ 617   \   \\   
		
		\hline
		
		\multirow{2}{*}{$N$}     & 5 if $(\frac{p}{3})=1$ \ \ &\multirow{2}{*}{2} & 1 if $(\frac{p}{7})=1$  & \multirow{2}{*}{1}  &\multirow{2}{*}{1}&\multirow{2}{*}{1}&\multirow{2}{*}{1}\\
		
		& 6 if $(\frac{p}{3})=-1$  &  &  2 if $(\frac{p}{7})=-1$   & && &\\

		\hline
		
		$m$ & 298 & 13 &  2 & 1 &1 &0&0\\
		
		\hline
		
	\end{tabular}
\end{center}
\caption{k=20}   

\vspace{1.2cm}

\makeatletter\def\@captype{table}\makeatother

\begin{center}
	\begin{tabular}{|c|c|c|c|c|c|c|c|c|}
		\hline
		$\ell$   & \  3 \ & \ 5 \  & \   7  \  &  \ 13  \  & \ 17  \ &  \ 131 \ & \ 593   \   \\  
		
		\hline
		
		\multirow{2}{*}{$N$}    & 6 if $(\frac{p}{3})=1$ \ \ &   \multirow{2}{*}{2}   &    \multirow{2}{*}{2}     &  \multirow{2}{*}{1}    &  \multirow{2}{*}{1}  &  \multirow{2}{*}{1}   &  \multirow{2}{*}{1}  \\
		
		& 7 if $(\frac{p}{3})=-1$  &  & &   &&&\\
		
		\hline
		
		$m$ & 18 & 14 &  37 & 1 &1 &0&0\\
		
		\hline
		
	\end{tabular}
\end{center}
\caption{k=22}

\vspace{1.2cm}

\makeatletter\def\@captype{table}\makeatother

\begin{center}
	\begin{tabular}{|c|c|c|c|c|c|c|c|c|}
		\hline
		$\ell$   & \  3 \ & \ 5 \  & \   7  \  &  \ 11  \  & \ 17  \ &  \ 19 \ & \ 657931   \   \\   \hline
		
		\multirow{2}{*}{$N$}   & 5 if $(\frac{p}{3})=1$ \ \ &\multirow{2}{*}{2} & 1 if $(\frac{p}{7})=1$  & \multirow{2}{*}{2}  &\multirow{2}{*}{1}&\multirow{2}{*}{1}&\multirow{2}{*}{1}\\
		& 6 if $(\frac{p}{3})=-1$  &  & 2 if $(\frac{p}{7})=-1$   & &&&\\
		\hline
		
		$m$ & 340 & 6 &  2 & 1 &1 &1&0\\
		
		\hline
		
	\end{tabular}
\end{center}
\caption{k=26}

\vspace{1.2cm}

For the case of type (ii), it can be shown that there are two exceptional primes for $\Delta_k$, i.e., $\ell=23$ when $k=12$ and $\ell=31$ when $k=16$. In fact we also have the explicit congruences in these cases. It is well-known that $a_p(\Delta_{12})$ satisfies the following congruences (see \cite{wilton}):
\begin{center}
	$\begin{array}{lll}
	a_p(\Delta_{12})\equiv 0& \mod 23 & if \  (\frac{p}{23})=-1,\\
	a_p(\Delta_{12})\equiv 2& \mod 23 & if \ p=u^2+23v^2 \ for \ integers \ u\ne 0, v,\\
	a_p(\Delta_{12})\equiv -1& \mod 23 & \ for  \ other \  p\ne 23.\\
	\end{array}$
\end{center}
The results of $\ell=31$  for $\Delta_{16}$ is quite similar and we will prove them by an idea proposed in \cite{sd1}.

\begin{prop} \label{congruencesfor31}
	Let $\Delta_{16}$ denote the unique cusp form of level $1$ and weight $16$. Let $a_p(\Delta_{16})$ be the $p$-th coefficient of $\Delta_{16}$. Then we have
	\begin{center}
		$\begin{array}{lll}
		a_p(\Delta_{16})\equiv 0& \mod 31 & if \  (\frac{p}{31})=-1,\\
		a_p(\Delta_{16})\equiv 2& \mod 31 & if \ p=u^2+31v^2 \ for \ integers \ u\ne 0, v,\\
		a_p(\Delta_{16})\equiv -1& \mod 31 & for  \ other \  p\ne 31.\\
		\end{array}$
	\end{center}
	Here $(\frac{\cdot}{\cdot})$ is the Legendre symbol.
\end{prop}

\begin{proof}
	Let $K$ be the quadratic field $\mathbb{Q}(\sqrt{-31})$. Then the integer ring of $K$ is $O_{K}=\{ \frac{a}{2}+\frac{b}{2}\sqrt{-31}:a\equiv b \bmod 2  \}$. It followed from \cite[Theorem 1.31]{ono} that 
	$$f=\sum\sum q^{m^2+mn+8n^2}-\sum\sum q^{2m^2+mn+4n^2},  \ \  where \ q=e^{2\pi iz},  $$
	is a modular form of weight $1$ with respect to $\Gamma_0(31)$. 
	
	By \cite[Theorem 11]{serresd1}, we know that $f^2\in S_2(\Gamma_{0}(31))$ is congruent mod $31$ to a modular form $g$ in $S_{32}(SL_2(\mathbb{Z}))$ that has integral coefficients, i.e., $f^2\equiv g \mod 31$. Sturm's Theorem \cite[Theorem 2]{sturm} implies 
	\begin{displaymath} 
	(2\Delta_{16})^2 \equiv f^2 \mod 31,
	\end{displaymath}
	by checking for the coefficients of $q^0, q^1$, and $q^2$. Therefore we have 
	\begin{equation} \label{delta16}
	2\Delta_{16} \equiv f \mod 31,
	\end{equation}
	since the coefficients of $q^1$ in each side are equal to $2$.
	
	If $p$ is a quadratic residue modulo $31$, then $p$ splits in $K$ as a product of two principal prime ideals or a product of non-principal ideals.
	
	If $p=\wp\bar{\wp}$ where $\wp$ is principal, we have
	
	$$ \wp=(\frac{a}{2}+\frac{b}{2}\sqrt{-31}) \ for \ some \ a \ and \ b \ with \ a\equiv b \bmod 2.$$
	Since $p$ is equal to the norm $N(\wp)$ of $\wp$, we know
	$$p=\frac{1}{4}(a^2+31b^2) \ with \ a-b \equiv 0 \bmod 2.$$
	Then
	$$a^2-b^2=4(p-8b^2).$$
	If both of $a$ and $b$ are odd, it follows that $8 \ | \ (a^2-b^2)$ and then $2|p$. This is  a contradiction. Thus $a$ and $b$ are both even. Therefore $u=\frac{a}{2}$ and $v=\frac{b}{2}$ are integers and $p=u^2+31v^2.$\\
	Set $m=u-v$ and $n=2v$, and then we have
	\begin{equation} \label{expressionofp1}
	p=m^2+mn+8n^2.
	\end{equation}
	It is clear that the steps above can be reversed and it follows that $p$ splits as a product of two principal ideals if and only if $p=u^2+31v^2$, if and only if there exist integers $m$ and $n$ such that $p=m^2+mn+8n^2$.\\
	For any such $p$, the equation (\ref{expressionofp1}) has exactly four solutions, two of which correspond to  $ \wp=(\frac{a}{2}+\frac{b}{2}\sqrt{-31})$ and the others correspond to $ \bar{\wp}=(\frac{a}{2}-\frac{b}{2}\sqrt{-31})$. This implies $a_p(f) = 4$ and from (\ref{delta16}) we have $ a_p(\Delta_{16})\equiv 2 \mod 31$ for all such $p$. 
	
	Now we suppose that $p=\wp\bar{\wp}$ where $\wp$ is non-principal. Let $\pi_2$ be a prime factor of the ideal $(2)$ and $\pi_3$ a prime factor of the ideal $(3)$. For an ideal $\mathcal{I}$ of $K$, we denote the ideal class of $\mathcal{I}$ by $[\mathcal{I}]$. Then it can be shown that the class number of $K$ is $3$ and $[\pi_2]=[\pi_3]^{-1}$ is an generator of the class group of $K$. This implies that either $\wp \pi_2$ or $\wp \pi_3$ is principal. \\
	If $\wp \pi_2$ is principal, we have
	$$   \wp\pi_2=(\frac{a}{2}+\frac{b}{2}\sqrt{-31}) \ for \ some \ a \ and \ b \ with \ a\equiv b \bmod 2.  $$
	So
	\begin{equation} \label{normofp}
	2p=N(\wp\pi_2)=\frac{1}{4}(a+b\sqrt{-31})(a-b\sqrt{-31}).
	\end{equation}
	Since $\wp=(p,\sqrt{-31}-r)$ and $\pi_2=(2,\frac{a}{2}-\frac{b}{2}\sqrt{-31})$ where $r$ is a root of equation $r^2+31\equiv 0 \bmod p$, there exists $\frac{x}{2}+\frac{y}{2}\sqrt{-31}\in O_{K}$ such that
	\begin{equation} \label{ptimes}
	p\cdot \frac{1+\sqrt{-31}}{2}=(\frac{a}{2}+\frac{b}{2}\sqrt{-31})\cdot (\frac{x}{2}+\frac{y}{2}\sqrt{-31}).
	\end{equation}
	From (\ref{normofp}) and (\ref{ptimes}), it follows that $a-b=4y$. This proves $4 \ | \ (a-b).$\\
	Set $m=\frac{a-b}{4}$ and $n=b$ and then we have
	\begin{equation} \label{expressionofp2}
	p=2m^2+mn+4n^2.
	\end{equation}
	If $\wp \pi_3$ is principal, we can consider $\bar{\wp}$ instead. Also we can obtain the following equations
	$$  3p=N(\wp\pi_2)=\frac{1}{4}(a+b\sqrt{-31})(a-b\sqrt{-31}), $$
	and
	$$ p\cdot \frac{1+\sqrt{-31}}{2}=(\frac{a}{2}+\frac{b}{2}\sqrt{-31})\cdot (\frac{x}{2}+\frac{y}{2}\sqrt{-31}). $$
	This implies $6 \ | \ (a-b).$\\
	Set $n=\frac{a-b}{6}$ and $m=b$ and then we also have $p=2m^2+mn+4n^2.$\\
	We reverse the above steps and then  prove that $p$ splits as a product of two non-principal ideals if and only if there exist integers $m$ and $n$ such that $p=2m^2+mn+4n^2$.\\
	For any such $p$, there are exactly two solutions for (\ref{expressionofp2}). This implies $a_p(f) =-2$ and from (\ref{delta16}) we have $ a_p(\Delta_{16})\equiv -1 \mod 31$ for any such $p$.  
	
	If $p$ is a quadratic non-residue modulo $31$, then $p$ is inert and the $p$-th coefficient $a_p(f)$ vanishes, which implies $ a_p(\Delta_{16})\equiv 0 \mod 31$. This completes the proof.

\end{proof}

\section{The bound of $n$ with $a_n(\Delta_{k})\ne0$}

As we discuss in the previous sections, we can now compute the bounds of $n$ for which $a_n(\Delta_{k})\ne0$ with $k=16, 18, 20, 22$ and $26$. It follows from Corollary \ref{primesfordelta} that we can consider only on the prime numbers rather than all the integers $n$ in each of the cases. \\ 

If $a_p(\Delta_{16})=0$ for a prime $p$, the congruences (\ref{congruenceformula}) with  $N$ and $m$ in  Table $1$, as well those in Proposition \ref{congruencesfor31} shows $p$ must satisfy

\begin{center}
	$\begin{array}{lll}
	p\equiv -1 \mod 5^{2}\cdot 7^{3}\cdot 11\cdot 3617, \\
	p\equiv  -1, 80, 161, 242, 323 ,404, 485, 566, 647 \mod 3^6 \ \ \ \ \ \ \ \mathrm{and} \\
	p\equiv \ a \ non$-$square \mod 31. \\
	\end{array}$
\end{center}
As discussed in the end of Section \ref{sec:representations}, the modular polynomials computed in \cite{1}\cite{martin} can be used to verify whether  $a_p(\Delta_{16})\equiv 0\mod 17\cdot 19\cdot 23\cdot 29\cdot 43$. \\
Moreover, in \cite{sd1} Swinnerton-Dyer shows that the fixed field $K$ of the kernal of the projective representation
\begin{center}
	$\tilde{\rho}_{\Delta_{16,59}}: Gal(\overline{\mathbb{Q}}|\mathbb{Q}) \rightarrow PGL_{2}(\overline{\mathbb{F}}_{59})$
\end{center}
is ramified only at $59$ over $\mathbb{Q}$ and is a normal extension with Galois group isomorphic to $S_4$. It follows that $K$ in fact is the splitting filed of the polynomial
$$\tilde P_{\Delta_{16,59}} = x^4-x^3-7x^2+11x+3.$$
This polynomial can be used to verify whether $a_p(\Delta_{16})\equiv 0\mod59$ as the same discussion in the end of Section \ref{sec:representations}.\\
Thus we now can do a systematically searching for the smallest prime $p$ in the congruences above, as well as $a_p(\Delta_{16})\equiv 0\mod 17\cdot 19\cdot 23\cdot 29\cdot 43\cdot 59$. Then we have
\begin{cor}
	The coefficients $a_n(\Delta_{16})$ is non-vanishing for all  $n$ with
	\[ n<12604744061516618549. \]
\end{cor}
\begin{proof}
	
\end{proof}

For a prime $p$, the congruences (\ref{congruenceformula}) with  $N$ and $m$ in  Table $2-5$ imply:
if $a_p(\Delta_{18})=0$,  then $p$ satisfies
\begin{center}
	$\begin{array}{lll}
	p\equiv -1 \mod 3^{6}\cdot 5^{3}\cdot 43867, \\
	p\equiv  -1, 19, 31 \mod 7^2, \ \ \ \ \ \ \ \ \ \\
	p\equiv  -1, 40, 94, 112, 118 \mod 11^2 \ \ \ \ \ \ \ \mathrm{and} \\
	p\equiv  -1, 4, 10 \mod 13; \ \ \ \ \ \ \ \\
	\end{array}$
\end{center}    
if $a_p(\Delta_{20})=0$,  then $p$ satisfies
\begin{center}
	$\begin{array}{lll}
	p\equiv -1 \mod 3^{6}\cdot 5^{2}\cdot 11\cdot 13\cdot 283\cdot 617 \ and \\
	p\equiv  -1, 19, 31 \mod 7^2; \ \ \ \ \ \ \ \ \ \\
	
	\end{array}$
\end{center}
if $a_p(\Delta_{22})=0$,  then $p$ satisfies
\begin{center}
	$\begin{array}{lll}
	p\equiv -1 \mod  5^{2}\cdot 7^2\cdot 13\cdot 17\cdot 131\cdot 593 \ and \\
	p\equiv  -1, 728, 1457 \mod 3^7; \ \ \ \ \ \ \ \ \ \\
	
	\end{array}$
\end{center}
and if $a_p(\Delta_{26})=0$,  then $p$ satisfies
\begin{center}
	$\begin{array}{lll}
	p\equiv -1 \mod 3^{6}\cdot 5^{2}\cdot 11\cdot 17\cdot 19, &  \\
	p\equiv  -1, 157780, 578462, 610260, 627364 \mod 657931 \ and \ \ \ \ \\
	p\equiv \ a \ non$-$square \mod 7. \\& \ \\
	\end{array}$
\end{center}
Likewise, the congruences above and the modular polynomials computed in \cite{1}\cite{martin} allow us to calculate the bounds $B_k$ of $n$:
\begin{prop}
	Let the pair ($k, B_k$) takes the values as in the following table. Then the coefficients $a_n(\Delta_{k})$ is non-vanishing for all $n$ with
	\[ n<B_k. \]
	
	\makeatletter\def\@captype{table}\makeatother
	
	\begin{center}
		\begin{tabular}{|c|l|}
			\hline
			$k$  & $B_k$   \\   
			
			\hline
			
			\ \  $16$ \ \ & \ $12604744061516618549$  \  \\
			\hline
			\ \  $18$ \  \ &  \  $143412400182350051864999$  \  \\
			\hline
			\ \ $20$ \ \ & \ $74201833676082662549 $ \  \\
			\hline
			\ \  $22$ \ \ & \ $28265095927027650599999$     \  \\
			\hline
			\ \   $26$ \ \  &  \ $818406791865712833299$  \  \\
			\hline
			
		\end{tabular}
	\end{center}

\end{prop}

\begin{proof}
	
\end{proof}

\section*{Acknowledgements}

The authors thank Ren\'{e} Schoof who inspires us to write this paper. Thanks go to J-P. Serre who provides with the polynomial for the case of $k=16$ and $\ell=59$. The first author is supported by The Fundamental Research Funds for the Central Universities (NO:222201514319). This work is supported by NSFC (NOs: 11171141, 11471154, 61403084).

\end{document}